\documentclass[12pt, reqno]{amsart}

\usepackage{amsmath}
\usepackage{amssymb,amsfonts,amscd}
\usepackage{latexsym}
\usepackage{mathrsfs}
\usepackage{color}

\usepackage{hyperref}

\newcommand\al{\alpha}
\newcommand\be{\beta}
\newcommand\ga{\gamma}
\newcommand\Ga{\Gamma}

\newcommand\De{\Delta}

\newcommand\la{\lambda}

\newcommand\si{\sigma}

\newcommand\na{\wt{\mathrm{m}}}

\newcommand\R{\mathbb R}

\newcommand\Z{\mathbb Z}
\newcommand\Y{\mathbb Y}
\newcommand\F{\mathbb F}

\newcommand\Sym{\operatorname{Sym}}

\newcommand\her{{\operatorname{her}}}

\newcommand\ex{\operatorname{ex}}

\newcommand\GL{\mathrm{GL}}
\newcommand\U{\mathrm{U}}
\newcommand\even{{\mathrm{even}}}
\newcommand\odd{{\mathrm{odd}}}
\newcommand\row{{\mathrm{row}}}
\newcommand\col{{\mathrm{col}}}
\newcommand\HL{{\mathrm{HL}}}

\newcommand\YHL{\wt\Y^{\mathrm{HL}}}
\newcommand\rel{\nearrow\!\!\!\nearrow}

\newcommand\conv{\circledast}

\newcommand\oo{\mathfrak o}

\newcommand\wt{\widetilde}

\newcommand\ccdot{\,\cdot\,}
\newcommand\wh{\widehat}

\newtheorem{theorem}{Theorem}[section]
\newtheorem{proposition}[theorem]{Proposition}
\newtheorem{lemma}[theorem]{Lemma}

\newtheorem{corollary}[theorem]{Corollary}

\theoremstyle{definition}
\newtheorem{definition}[theorem]{Definition}
\newtheorem{remark}[theorem]{Remark}

\newtheorem{problem}[theorem]{Problem}

\numberwithin{equation}{section}

\begin{document}

\title[]{Hall-Littlewood-positive harmonic functionals on the algebra of symmetric functions}

\author{Cesar Cuenca and Grigori Olshanski}

\date{}

\thanks{}

\begin{abstract}
We study the problem of describing the set of real functionals on the quotient $\Sym/(p_2-1)$ of the ring of symmetric functions that are nonnegative on the images of certain modified Hall-Littlewood symmetric functions.
This question is equivalent to the problem, posed in \cite{CO1}, 
of describing the set of coadjoint-invariant measures for unitary groups over a finite field in the infinite-dimensional setting.
Our main results constitute partial progress towards this problem. Firstly, we show that the desired set of functionals is very large, in the sense that it contains explicit families of examples depending on infinitely many parameters.
Secondly, we provide an analogue of Kerov's mixing construction that produces new sought after functionals from known old ones.
This construction depends on an explicit ``$p_2$-twisted action'' of $\Sym$ on itself and the resulting dual map that makes $\Sym$ into a comodule.
Finally, our third main result explains the relation between the $p_2$-twisted comultiplication and the usual comultiplication on $\Sym$.
\end{abstract}

\maketitle


\section{Introduction}\label{sect1}

\subsection{}  Let $\Sym=\R[p_1,p_2,\dots]$ denote the algebra of symmetric functions over the field of real numbers; here, $p_1,p_2,\dots$ are the Newton power sums. Next, let $\Y$ denote the set of all partitions, which are identified with their Young diagrams, and let $t$ be a real parameter. We denote by $\{P_\la(\ccdot;t): \la\in\Y\}$  the homogeneous basis of $\Sym$ formed by the Hall-Littlewood symmetric functions with parameter $t$; see~\cite[Ch. III]{Mac}. We often use `HL' as a shorthand for `Hall-Littlewood'. 

\begin{definition}\label{def1.A}
Let $\varphi:\Sym\to\R$ be a linear functional. We fix $t\in(-1,1)$.

(i) We say that $\varphi$ is \emph{$p_1$-harmonic} if $\varphi(p_1f)=\varphi(f)$, for any $f\in\Sym$.

(ii) We say that $\varphi$ is \emph{$t$-HL-positive} if $\varphi(P_\la(\ccdot;t))\ge0$, for all $\la\in\Y$.

(iii) We denote by $\Phi(t)$ the set of all linear functionals which are both $p_1$-harmonic and $t$-HL-positive. Furthermore, we denote by  $\Phi_1(t)$ the subset of $\Phi(t)$ singled out by the additional normalization condition $\varphi(1)=1$.
\end{definition}

The space of linear functionals $\Sym\to\R$ is isomorphic to $\R^\infty$.
Taken separately, conditions (i) and (ii) determine two subsets whose geometric structure is simple: the second set looks like the cone $\R_{\ge0}^\infty$ in $\R^\infty$, while the first set is a linear subspace of $\R^\infty$. It is the combination of the two conditions that makes the picture nontrivial.

Note that $\Phi(t)$ is a convex cone and $\Phi_1(t)$ is a convex set that serves as a distinguished base of the cone $\Phi(t)$. We denote by $\ex(\Phi_1(t))$ the set of \emph{extreme points} of $\Phi_1(t)$. Then the functionals of the form $c\varphi$ with $\varphi\in\ex(\Phi_1(t))$ fixed and $c\ge0$ are the \emph{extreme rays} of the cone $\Phi(t)$.

The knowledge of $\ex(\Phi_1(t))$ serves to describe $\Phi_1(t)$ and $\Phi(t)$, because $\Phi_1(t)$ is a \emph{Choquet simplex}: it is isomorphic, in a natural way, to the convex set of probability measures on $\ex(\Phi_1(t))$ (see Proposition~\ref{prop2.A}). Likewise, under this isomorphism, the cone $\Phi(t)$ is realized as the set of finite measures on $\ex(\Phi_1(t))$.

\subsection{} 

A linear functional $\varphi:\Sym\to\R$ is said to be \emph{multiplicative} if $\varphi(1)=1$ and $\varphi(ab)=\varphi(a)\varphi(b)$, for any $a,b\in\Sym$. Because $\Sym$ is freely generated by the power-sum symmetric functions $p_1,p_2,\dots$, a multiplicative functional $\varphi$ is uniquely determined by its values $\varphi(p_k)$, $k=1,2,\dots$, and these values can be arbitrary real numbers.

\begin{theorem}[Kerov-Matveev]\label{thm1.A}
Fix $t\in(-1,1)$. The functionals $\varphi\in\ex(\Phi_1(t))$ are precisely those multiplicative functionals whose values on the generators $p_1,p_2,\dots$ are given by the formulas
\begin{equation}
\varphi(p_1)=1,\quad \varphi(p_k)=\sum_{i=1}^\infty \al_i^k+(-1)^{k-1}\frac{1}{1-t^k}\sum_{j=1}^\infty\be_j^k, \quad k=2,3,\dots,
\end{equation}
for some real parameters $\al_i$ and $\be_j$ that satisfy the following conditions:
\begin{equation}
\al_1\ge\al_2\ge\dots\ge0,\quad \be_1\ge\be_2\ge\dots\ge0,\quad \sum_{i=1}^\infty\al_i +\frac{1}{1-t}\sum_{j=1}^\infty\be_j\le1.
\end{equation}
\end{theorem}

This theorem is a special case of a more general result, which concerns the two-parameter family of bases of $\Sym$ consisting of the Macdonald symmetric functions; see Matveev \cite[Theorem~1.4 and Proposition~1.6]{Mat}. That result has a rich history; see \cite[Section~1]{Mat}.

\subsection{}

Let $q$ be a prime power, that is, $q:=p^m$, for some prime number $p$ and positive integer $m$. We denote by $\F_q$ the corresponding finite field of cardinality $|\F_q|=q$.

In~\cite{CO1}, we studied the set of measures on certain space of infinite matrices over $\F_q$ that are invariant with respect to the action of the group
$$
\GL(\infty,\F_q):=\varinjlim \GL(n,\F_q)=\bigcup_{n=1}^\infty \GL(n,\F_q),
$$
known as the \emph{coadjoint action}.\footnote{Our work was inspired by the papers of Vershik-Kerov~\cite{VK1},~\cite{VK2}, and Gorin-Kerov-Vershik~\cite{GKV}; see the introduction in~\cite{CO1}.}

We showed that the problem of describing these measures is reduced to the description of the functionals $\varphi\in\Phi_1(t)$ with $t=q^{-1}$. Because $q^{-1}\in(0,1)$, the solution is then provided by Theorem \ref{thm1.A}.

Next, also in~\cite{CO1}, we posed the analogous problem for the two twin infinite-dimensional unitary groups
\begin{equation}\label{eq1.B}
\U(2\infty,\F_{q^2}):=\varinjlim \U(2n,\F_{q^2}), \quad \U(2\infty+1,\F_{q^2}):=\varinjlim \U(2n+1,\F_{q^2}).
\end{equation}
It turned out that, again, describing the invariant measures for the coadjoint actions of these groups can be reduced to a problem about positive harmonic functionals, and that problem is related to the Hall-Littlewood symmetric functions. However, the very notions of positivity and harmonicity needed to be substantially modified. We proceed to the exact formulations.

\subsection{}\label{sec:1.4}

From now on, we assume that $t\in(0,1)$. We define the \emph{modified Hall-Littlewood symmetric functions} with parameter $-t$ by
\begin{equation}\label{eq1.A}
\wt P_\la(\ccdot;-t):=(-1)^{n(\la)}P_\la(\ccdot;-t),\text{ for all $\la\in\Y$, where $n(\la):=\sum_{i\ge 1}{(i-1)\la_i}$}.
\end{equation}

\begin{definition}[cf.~Definition~\ref{def1.A}]\label{def1.B}
Let $t\in(0,1)$ be fixed, and let $\psi:\Sym\to\R$ be a linear functional.

(i) We say that $\psi$ is \emph{$p_2$-harmonic} if $\psi(p_2f)=\psi(f)$, for any $f\in\Sym$.

(ii) We say that $\psi$ is \emph{$(-t)$-HL-positive} if $\psi(\wt P_\la(\ccdot;-t))\ge0$, for all $\la\in\Y$. 

(iii) We denote by $\Psi(-t)$ the convex cone formed by all linear functionals on $\Sym$ which are both $p_2$-harmonic and $(-t)$-HL-positive.
\end{definition}

As shown in our paper~\cite{CO1}, the problem of describing the set of coadjoint-invariant measures for the unitary groups \eqref{eq1.B} is reduced to the description of the cone $\Psi(-t)$ for the special values $t=q^{-1}$, where $q$ is an odd prime power. Thus, we come to the following problem.

\begin{problem}\label{problem1.A}
Study the convex cones $\Psi(-t)$, $t=q^{-1}$. In particular, describe their extreme rays.
\end{problem}

The direct sum decomposition $\Sym=\Sym_\even\oplus\Sym_\odd$ by parity of degree naturally leads to the decomposition
$$
\Psi(-t)=\Psi_\even(-t)\oplus\Psi_\odd(-t),
$$
which reduces our main problem to the subproblems of describing $\Psi_\even(-t)$ and $\Psi_\odd(-t)$.
We envision the definite solutions to resemble what Theorem~\ref{thm1.A} and Proposition~\ref{prop2.A} are to the problem of describing the cone $\Phi(t)$.
Unfortunately, we believe that this characterization of $\Phi(t)$, as proved in \cite{Mat}, cannot be easily replicated to describe $\Psi_\even(-t)$ and $\Psi_\odd(-t)$; see the heuristic argument in Section~\ref{sect7}.
Hence, it appears that new ideas are necessary.

\subsection{}

Our first main result is the description of explicit cone embeddings
$$
\Phi(t^2)\hookrightarrow\Psi_\even(-t),\qquad\Phi(t^2)\hookrightarrow\Psi_\odd(-t).
$$
Hence, the cones $\Psi_\even(-t)$ and $\Psi_\odd(-t)$ are at least as large as $\Phi(t^2)$, thus showing that Problem~\ref{problem1.A} has a nontrivial answer.
The precise embeddings are shown in Theorems~\ref{thm4.A} and \ref{thm4.B}.

The second main result is an adaptation of \emph{Kerov's mixing construction}, that is, in Theorem~\ref{thm5.A} and Corollary~\ref{cor5.A}, we find maps
\[
\Phi(t^2)\times\Psi_\even(-t)\times [0,1]\to\Psi_\even(-t),\qquad\Phi(t^2)\times\Psi_\odd(-t)\times [0,1]\to\Psi_\odd(-t)
\]
that can be used to construct new functionals in $\Psi_\even(-t)$ and $\Psi_\odd(-t)$ from old ones.
The mixing construction depends on the ``$p_2$-twisted action'' $\na$ of $\Sym$ on itself (see equation~\eqref{eq3.C}) and its dual map $\wt\De$.

Our third and final main result is Theorem~\ref{thm6.A} that proves a relationship between $\wt\De$ and the usual comultiplication map of $\Sym$.
This theorem describes the action of $\Sym$ on $\Sym^{\otimes 2}$ that makes $\wt\Delta:\Sym\to\Sym^{\otimes 2}$ (with the domain endowed with the $p_2$-twisted action) into a $\Sym$-module homomorphism.

\subsection{}

In Section~\ref{sect2}, we recall some generalities on branching graphs.
Next, in Section~\ref{sect3}, we turn our focus to the branching graphs $\YHL_\even(-t)$, $\YHL_\odd(-t)$, and deduce from a recent result of Shen and Van Peski~\cite{SVP1} that the structure constants of certain $\Sym$-modules are nonnegative; this fact will be used in the proofs of our main theorems in the following two sections.
Section~\ref{sect4} proves our first main result, split into Theorem~\ref{thm4.A} and \ref{thm4.B}: both cones $\Psi_{\even}(-t)$ and $\Psi_{\odd}(-t)$ are at least as large as $\Phi(t^2)$.
As it turns out, these results are limiting cases of an adaptation of Kerov's mixing construction, described in Section~\ref{sect5}, that produces new functionals in $\Psi(-t)$ from old ones; the precise statement is Theorem~\ref{thm5.A} and constitutes our second main result.
Section~\ref{sect6} contains our third main result in Theorem~\ref{thm6.A} and explains its relationship to our note~\cite{CO2} and to van Leeuwen's work~\cite{L}.
We conclude with some remarks in Section~\ref{sect7}.

\subsection{Acknowledgements}

The authors are grateful to Jiahe Shen and Roger Van Peski for helpful discussions.
C.C. was supported by the NSF grant DMS-2348139 and the Simons Foundation’s Travel Support for Mathematicians grant MP-TSM-00006777.
G.O. was supported by the Ministry of Science and Higher Education of the Russian Federation GZ project.

\section{Preliminaries: branching graphs and Kerov's mixing construction}\label{sect2}

\subsection{}

We start with a general formalism (see, e.g., \cite[\S2]{BO}, \cite[Ch.~1]{Ke2} and references therein).

\begin{definition}\label{def2.B}
A \emph{branching graph} is a $\Z_{\ge0}$-graded, connected, rooted graph $\Ga=(V,E)$ without pending vertices, and where each edge $e\in E$ is endowed with a weight $w(e)$, which is a strictly positive real number.
\end{definition}

In more detail, the vertex set $V$ is partitioned into levels indexed by the nonnegative integers: $V=V_0\sqcup V_1\sqcup V_2\sqcup\cdots$. The $0$-th level $V_0$ consists of a single vertex, denoted by $\varnothing$, that serves as the root of $\Ga$. The edges join vertices of adjacent levels only. Moreover, each vertex $v\in V_n$, $n>0$, is joined with at least one vertex from $V_{n-1}$ and at least one vertex from $V_{n+1}$.

\begin{definition}\label{def2.C}
Let $\Ga=(V,E)$ be a branching graph. We denote by $\Phi(\Ga)$ the set of functions $\varphi: V\to\R_{\ge0}$, which are \emph{harmonic} in the sense that for any $n\in\Z_{\ge0}$ and any vertex $v\in V_n$, one has
\begin{equation}\label{eq2.G}
\varphi(v)=\sum_{v'\in V_{n+1}: \,(v,v')\in E} w(v,v') \varphi(v').
\end{equation}
Next, we denote by $\Phi_1(\Ga)$ the subset of functions $\varphi\in\Phi(\Ga)$ satisfying the additional normalization condition $\varphi(\varnothing)=1$.
\end{definition}

Evidently, $\Phi(\Ga)$ is a convex cone and $\Phi_1(\Ga)$ is a convex set. We denote by $\ex(\Phi_1(\Ga))$ the set of extreme points of $\Phi_1(\Ga)$.

\begin{proposition}\label{prop2.A}
Let, as above, $\Ga=(V,E)$ be a branching graph. We additionally assume that all levels $V_n$ are finite sets.

{\rm(i)} The set $\Phi_1(\Ga)$ is nonempty and it serves as a base of the cone $\Phi(\Ga)$.

{\rm(ii)} The subset $\ex(\Phi_1(\Ga))\subset \Phi_1(\Ga)$ has a natural Borel structure.

{\rm(iii)} There is a natural isomorphism of convex sets between $\Phi_1(\Ga)$ and the set of probability Borel measures on $\ex(\Phi_1(\Ga))$.

{\rm(iv)} Likewise, the cone $\Phi(\Ga)$ is isomorphic to the cone of finite Borel measures on $\ex(\Phi_1(\Ga))$.
\end{proposition}

\begin{proof}
For each vertex $v\in V_n$, $n>0$, there exists a monotone path in the graph, joining $v$ with the root, that is, a sequence $(v_0,v_1,\dots,v_n)$, where $v_i\in V_i$, $(v_i, v_{i+1})\in E$, $v_0=\varnothing$, and $v_n=v$. 
Assume that $\varphi\in\Phi(\Ga)$ is such that $\varphi(\varnothing)=0$. Using induction on $n$, we deduce from \eqref{eq2.G} that $\varphi(v)=0$, for all $v\in V_n$, and all $n>0$. 
This proves that $\Phi_1(\Ga)$ serves as a base of the cone $\Phi(\Ga)$ provided we already know that $\Phi_1(\Ga)$ is nonempty. 

We are going to show that $\Phi_1(\Ga)$ can be represented as a projective limit of certain finite-dimensional simplices $\triangle_n$. It will imply that $\Phi_1(\Ga)$ is nonempty (since any projective limit of compact sets is nonempty). Further, the claims (ii)-(iv) will follow from Choquet's theorem: see \cite[Theorem~9.2]{Ols1} or Winkler \cite[Theorem~3.2.3]{W}.

We introduce a function $d:V\to\R_{>0}$ as follows: $d(\varnothing):=1$, and for $v\ne\varnothing$, we define $d(v)$ as a weighted sum over all monotone paths from $\varnothing$ to $v$, where the weight of a path is the product of the weights of the edges constituting the path. Then the basic equations \eqref{eq2.G} can be rewritten as follows:
\begin{equation}\label{eq2.H}
d(v)\varphi(v)=\sum_{v'\in V_{n+1}: \,(v,v')\in E}\frac{d(v)w(v,v')}{d(v')} d(v')\varphi(v'), \quad v\in V_n, \; n\ge0.
\end{equation}

The key observation is that for any $v'\in V_{n+1}$, where $n=0,1, 2,\dots$, one has
\begin{equation}\label{eq2.I}
\sum_{v\in V_n: \,(v,v')\in E}\frac{d(v)w(v,v')}{d(v')}=1.
\end{equation}
From \eqref{eq2.I} and \eqref{eq2.H} it follows that for any $\varphi\in\Phi(\Ga)$, the sum 
$
\sum\limits_{v\in V_n} d(v)\varphi(v)
$
is independent of $n$. Hence, in particular, if $\varphi\in\Phi_1(\Ga)$, then this sum is equal to $d(\varnothing)\varphi(\varnothing)=1$.

For $n=1,2,\dots$, we define $\triangle_n$ as the abstract simplex with vertices $v\in V_n$: this means that points $x_n\in\triangle_n$ are formal convex combinations of the form
$$
x_n=\sum_{v\in V_n} c(v) v,\text{ where $c(v)\ge 0$ and $\sum_{v\in V_n}c(v)=1$}. 
$$
By virtue of \eqref{eq2.I}, for each $n=1,2,\dots$, there exists an affine map $\pi_n:\triangle_n\to\triangle_{n-1}$ such that  
$$
\pi_n(v')=\sum_{v \in V_{n-1}: \,(v,v')\in E} \frac{d(v)w(v,v')}{d(v')}v,\,\text{ for all }v'\in V_n.
$$

Now, let $\varphi\in\Phi_1(\Ga)$ be arbitrary. Given $n=1,2,\dots$, we assign to $\varphi$ an infinite sequence
$$
(x_1,x_2,\dots)\in \triangle_1\times\triangle _2\times\dots
$$ 
by setting 
$$
x_n:=\sum_{v \in V_n}d(v)\varphi(v)v
$$
(note that $x_n$ is in $\triangle_n$ because $\sum\limits_{v \in V_n}d(v)\varphi(v)=1$). 

The relations \eqref{eq2.H} precisely mean that $x_{n-1}=\pi_n(x_n)$ for each $n$, therefore $(x_1, x_2,\dots)$ is an element of the projective limit space $\varprojlim(\triangle_n,\pi_n)$.  And vice versa, any element of $\varprojlim(\triangle_n,\pi_n)$ comes from a (unique) function $\varphi\in\Phi_1(\Ga)$. 

This gives us the desired bijection $\Phi_1(\Ga) \leftrightarrow \varprojlim(\triangle_n,\pi_n)$. 
\end{proof}

\subsection{}

We use the notation of Macdonald~\cite[Ch. III, \S3]{Mac} for the structure constants $f^\la_{\mu\nu}(t)$ of the algebra $\Sym$ in the HL basis $\{P_\la(\ccdot;t)\colon\la\in\Y\}$:
\begin{equation}\label{eq2.D}
P_\mu(\ccdot;t)P_\nu(\ccdot;t)=\sum_{\la:\, |\la|=|\mu|+|\nu|}f^\la_{\mu\nu}(t)P_\la(\ccdot;t).
\end{equation}

For $\mu,\la\in\Y$, we write $\mu\nearrow\la$ if $\mu\subset\la$ and $|\la|=|\mu|+1$, that is, if $\la$ is obtained from $\mu$ by adding a box. In the special case of \eqref{eq2.D} corresponding to $\nu=(1)$, we have that $P_{(1)}(\ccdot;t)=p_1$, and therefore \eqref{eq2.D} can be written in the form
\begin{equation}\label{eq2.E}
p_1\cdot P_\mu(\ccdot;t)=\sum_{\la:\, \mu\nearrow\la} w(\mu\nearrow\la;t)P_\la(\ccdot;t),
\end{equation}
where, denoting by $j$ the column number of the box $\la/\mu$, 
\begin{equation}\label{eq2.F}
w(\mu\nearrow\la;t):=f^\la_{\mu(1)}(t)=\frac{1-t^{\la'_j-\la'_{j+1}}}{1-t},
\end{equation}
see \cite[Ch. III, (3.2)]{Mac}. 

\begin{definition}\label{def2.D}
Let $t\in(-1,1)$. The \emph{HL-deformed Young graph with parameter $t$}, denoted by $\Y^\HL(t)$, is the branching graph  $(V,E)$ with the vertex set $V=\Y$ and its natural grading by size $\Y=\Y_0\sqcup\Y_1\sqcup\Y_2\sqcup\dots$; further,  the edge set $E$ is formed by the pairs $\mu\nearrow\la$, and the edge weights are the quantities $w(\mu\nearrow\la;t)$ given by \eqref{eq2.F}.
\end{definition}

This definition is a special case of a branching graph, in the sense of Definition~\ref{def2.B}. Indeed, the underlying graph is the conventional Young graph (with vertex set $\Y$ and edges $\mu\nearrow\la$), which is evidently connected and has no pending vertices, while the edge weights are strictly positive, as seen from \eqref{eq2.F}.
As a result, Proposition~\ref{prop2.A} is applicable to $\Y^\HL(t)$.

We note that in our previous paper~\cite{CO1}, the branching graph $\Y^\HL(t)$ is defined somewhat differently, in terms of the HL functions $Q_\la(\ccdot;t)$ instead of the $P_\la(\ccdot;t)$'s (see \cite[(4.5) and Definition 4.7]{CO1}). 
However, this does not affect the definition of the cone $\Phi(t)$, because $Q_\la(\ccdot;t)=b_\la(t)P_\la(\ccdot;t)$, where the factor $b_\la(t)$ is strictly positive for any $\la\in\Y$: indeed, by writing $\la=(1^{m_1(\la)}2^{m_2(\la)}\dots)$, one has
\begin{equation}\label{eq2.B}
b_\la(t)=\prod_{i\ge1} (t;t)_{m_i(\la)},\text{ where $(t;t)_m:=(1-t)(1-t^2)\cdots(1-t^m)$},
\end{equation}
see~\cite[Ch.~III, (2.12)]{Mac}, showing that all factors are strictly positive whenever $t\in(-1,1)$.

\smallskip

Finally, from the comparison between Definition~\ref{def2.C} and Definition~\ref{def1.A}, it is clear that under the identification $\varphi(P_\la(\ccdot;t))=\varphi(\la)$, we have $\Phi(\Y^\HL(t))=\Phi(t)$ and $\Phi_1(\Y^\HL(t))=\Phi_1(t)$.
Hence, by virtue of Proposition~\ref{prop2.A}, the convex cone $\Phi(t)$ is isomorphic to the cone of finite Borel measures on $\ex(\Phi_1(t))$, and since Theorem~\ref{thm1.A} completely describes the set $\ex(\Phi_1(t))$, then $\Phi(t)$ is also completely described, for all $t\in(-1,1)$.

\subsection{}\label{sec2.3}

Recall the \emph{Cauchy identity} for HL functions \cite[Ch.~III, (4.4)]{Mac}: 
\begin{equation}\label{eq2.J}
\sum_{\la\in\Y} P_\la(x_1,x_2,\dots;t)Q_\la(y_1,y_2,\dots;t)=\prod_{i,j\ge 1}\dfrac{1-tx_iy_j}{1-x_iy_j}.
\end{equation}

We denote by $\De:\Sym\to\Sym^{\otimes2}$ the standard comultiplication map (\cite[Ch.~I, \S5, Ex.~25]{Mac}). From the multiplicativity of the right-hand side of \eqref{eq2.J}, it follows that the coefficients $f^\la_{\mu\nu}(t)$ serve also as the structure constants of $\De$ in the basis $\{Q_\la(\ccdot;t)\colon\la\in\Y\}$:
\begin{equation}\label{eq2.L}
\De\big(Q_\la(\ccdot;t)\big) = \sum_{\mu,\nu:\,|\mu|+|\nu|=|\la|}f^\la_{\mu\nu}(t)\, Q_\mu(\ccdot;t)\otimes Q_\nu(\ccdot;t).
\end{equation}

Given two linear functionals, $\varphi$ and $\psi$, on the algebra $\Sym$, we can form another linear functional, denoted by  $\phi*\psi$, as follows:
\begin{equation}\label{eq2.C}
(\phi*\psi)(a):=(\phi\otimes\psi)(\De(a)), \quad a\in\Sym.
\end{equation}

In another direction, given a linear functional $\varphi:\Sym\to\R$ and a real number $r\ge0$, we can define a new linear functional $\varphi_r$, called the \emph{dilation} of $\varphi$ with parameter $r$, by declaring its values on homogeneous elements $a\in\Sym$ to be
\begin{equation}\label{eq2.M}
\varphi_r(a):=\varphi(a)r^{\deg a}, 
\end{equation}
with the convention that $0^0:=1$.

\begin{definition}\label{def2.A}
Let $\varphi$ and $\psi$ be two linear functionals on $\Sym$. Their \emph{mixing with parameters $(r,s)$}, where $r\ge0$, $s\ge0$, $r+s=1$, is  the linear functional $\varphi_r*\psi_s$. 
\end{definition}

In more detail, the value of $\varphi_r*\psi_s$ on a homogeneous element $a\in\Sym$ is given by the following formula. Write any decomposition $\De(a)=\sum_i a'_i\otimes a''_i$, where $a'_i$ and $a''_i$ are homogeneous. Then
\begin{equation}\label{eq2.A}
(\varphi_r*\psi_s)(a)=\sum_i \varphi(a'_i)\psi(a''_i)r^{\deg a'_i}s^{\deg a''_i}.
\end{equation}

\begin{lemma}\label{lem2.A}
If $t\in(0,1)$, then $f^\la_{\mu\nu}(t)\ge0$, for all $\la,\mu,\nu$.
\end{lemma}

This follows from Schwer~\cite[Theorem 1.3]{S}, Ram~\cite[Theorem 4.9]{R}, or Yip~\cite[Theorem 4.13]{Y}, each of which expresses the quantities $f^\la_{\mu\nu}(t)$ as weighted sums indexed by various combinatorial objects with manifestly nonnegative weights, whenever $t\in(0,1)$.

We also point out that in the case $t=q^{-1}$, the lemma follows from \cite[Ch.~III, (3.5)]{Mac}, which shows that $q^{n(\la)-n(\mu)-n(\nu)}f^\la_{\mu\nu}(q^{-1})$ is the solution to an enumerative problem and therefore a nonnegative integer.

\medskip

Recall that $\Phi_1(t)$ denotes the set of linear functionals $\varphi:\Sym\to\R$ which are $p_1$-harmonic, $t$-HL-positive and normalized by the condition $\varphi(1)=1$.

\begin{proposition}\label{prop2.B}
Let $t\in(0,1)$ and let $r,s$ be nonnegative real numbers with $r+s=1$. If $\varphi,\psi\in\Phi_1(t)$, then $\varphi_r*\psi_s\in\Phi_1(t)$.
\end{proposition}

\begin{proof}
(i) Let us check that $\varphi_r*\psi_s$ is $p_1$-harmonic, that is,
\begin{equation}\label{eq:p1_harmonic}
(\varphi_r*\psi_s)(p_1a)=(\varphi_r*\psi_s)(a),\text{ for any }a\in\Sym.
\end{equation}
Without loss of generality we may assume that $a$ is homogeneous. Then we can write $\De(a)=\sum a'_i\otimes a''_i$, where all $a'_i$ and $a''_i$ are homogeneous, too. By definition, we then have that the right-hand side of \eqref{eq:p1_harmonic} is equal to
$$
(\varphi_r*\psi_s)(a)=\sum \varphi(a'_i)\psi(a''_i)r^{\deg a'_i}s^{\deg a''_i}.
$$
Likewise, using the fact that $\De$ is an algebra morphism, that $\varphi$ and $\psi$ are $p_1$-harmonic, and $\De(p_1)=p_1\otimes1+1\otimes p_1$, we obtain that the left-hand side of \eqref{eq:p1_harmonic} equals
\begin{gather*}
(\varphi_r*\psi_s)(p_1a)=\sum \varphi(a'_i)\psi(a''_i)(r^{1+\deg a'_i}s^{\deg a''_i}+r^{\deg a'_i}s^{1+\deg a''_i})\\
=(r+s)\sum \varphi(a'_i)\psi(a''_i)r^{\deg a'_i}s^{\deg a''_i}.
\end{gather*}
Because $r+s=1$, the equality \eqref{eq:p1_harmonic} follows.

(ii) Let us now check the nonnegativity condition. Recall that $Q_\la(\ccdot;t)=b_\la(t)P_\la(\ccdot;t)$ and that $b_\la(t)>0$, for all $\la\in\Y$ and $t\in(0,1)$ (see equation \eqref{eq2.B} above). Therefore, to examine nonnegativity, we may freely switch from the basis $\{P_\la(\ccdot;t) \colon \la\in\Y\}$ to the basis $\{Q_\la(\ccdot;t) \colon \la\in\Y\}$. From \eqref{eq2.L}, we have
$$
(\varphi_r*\psi_s)\big(Q_\la(\ccdot;t)\big) = \sum_{\mu,\nu:\,|\mu|+|\nu|=|\la|}f^\la_{\mu\nu}(t) \varphi\big(Q_\mu(\ccdot;t)\big)\psi\big(Q_\nu(\ccdot;t)\big)r^{|\mu|}s^{|\nu|}.
$$
All the factors in the right-hand side are nonnegative, due to Lemma~\ref{lem2.A} and the fact that both $\varphi$ and $\psi$ are $t$-HL-positive, so the final sum is also nonnegative, as desired.

(iii) Finally, because $\De(1)=1\otimes1$, we have $(\varphi_r*\psi_s)(1)=1$.
\end{proof}

The previous proposition can be generalized to an arbitrary number of linear functionals; let us begin with the following consequence of the associativity and commutativity of $\Delta$.

\begin{lemma}\label{lem2.B}
The binary operation $(\phi*\psi)(a)=(\phi\otimes\psi)(\De(a))$, defined in equation \eqref{eq2.C}, on the space $\Sym'$ of linear functionals $\Sym\to\R$ is both associative and commutative. 
Therefore $\Sym'$, with $*$ serving as the multiplication, is a commutative ring.
\end{lemma}

As a result of this lemma, for any integer $m\ge 1$ and linear functionals $\psi^{(1)},\dots,\psi^{(m)}$, we can unambiguously make sense of their product $\psi^{(1)}*\dots*\psi^{(m)}$.

\medskip

Next, note that for any $r,s\ge 0$, the dilations \eqref{eq2.C} compose in the sense that $(\varphi_r)_s=\varphi_{rs}$.
More generally, due to the fact that $\Delta$ is degree-preserving, we deduce that
\[
\big(\varphi^{(1)}_{r_1}*\dots*\varphi^{(m)}_{r_m}\big)_s = \varphi^{(1)}_{r_1s}*\dots\varphi^{(m)}_{r_ms},
\]
for any $m$ linear functionals $\varphi^{(1)},\dots,\varphi^{(m)}\in\Sym'$, and any $r_1,\dots,r_m,s\ge 0$.
A usual induction argument then leads to the following generalization of Proposition~\ref{prop2.B}.

\begin{proposition}[Kerov's mixing construction]\label{prop2.C}
Let $t\in(0,1)$, let $m\ge 2$ be an integer, and let $r_1,\dots,r_m$ be positive real numbers such that $r_1+\dots+r_m=1$. If $\varphi^{(1)},\dots,\varphi^{(m)}\in\Phi_1(t)$, then $\varphi^{(1)}_{r_1}*\dots*\varphi^{(m)}_{r_m}\in\Phi_1(t)$.
\end{proposition}

In more detail, we note that the value of $\varphi^{(1)}_{r_1}*\dots*\varphi^{(m)}_{r_m}$ on a homogeneous $a\in\Sym$ can be obtained as follows. 
If we write any decomposition $\Delta^m(a)=\sum_i{ a_i^{(1)}\otimes\dots\otimes a_i^{(m)} }$, where $a_i^{(1)},\dots,a_i^{(m)}$ are all homogeneous, then
\begin{equation}\label{eq2.O}
\big(\varphi^{(1)}_{r_1}*\dots*\varphi^{(m)}_{r_m}\big)(a) = \sum_i{ \varphi^{(1)}(a_i^{(1)})\cdots\varphi^{(m)}(a_i^{(m)})\,r_1^{\deg a_i^{(1)}}\cdots r_m^{\deg a_m^{(1)}} }.
\end{equation}

Another way to obtain functionals in $\Phi_1(t)$ is by taking limits.
For the next proposition, recall that for any $\mu=(1^{m_1(\mu)}2^{m_2(\mu)}\dots)\in\Y$, we denote $p_\mu:=\prod_{i\ge 1}{p_i^{m_i(\mu)}}$.

\begin{proposition}\label{prop2.D}
Let $t\in (0,1)$ and let $\{\varphi^{(n)}\}_{n\ge 1}\subset\Phi_1(t)$ be a sequence such that the limits $\lim_{n\to\infty}{\varphi^{(n)}(p_\mu)}$ exist, for all $\mu\in\Y$.
Then the unique linear functional $\varphi:\Sym\to\R$ defined by
\begin{equation}\label{eq2.N}
\varphi(p_\mu) := \lim_{n\to\infty}{\varphi^{(n)}(p_\mu)},\text{ for all }\mu\in\Y,
\end{equation}
and extended by linearity, belongs to $\Phi_1(t)$.
\end{proposition}
\begin{proof}
As $\{p_\mu : \mu\in\Y\}$ is a basis of $\Sym$, equations \eqref{eq2.N} imply that $\varphi(a) := \lim_{n\to\infty}{\varphi^{(n)}(a)}$, for all $a\in\Sym$.
Then the conditions of $p_1$-harmonicity and $t$-HL-positivity of $\varphi$ follow form the fact that they hold for all $\varphi^{(n)}$, and by taking the limit $n\to\infty$.
Also, $\varphi(1) = \lim_{n\to\infty}{\varphi^{(n)}(1)} = \lim_{n\to\infty}{1} = 1$. Hence, $\varphi\in\Phi_1(t)$.
\end{proof}

\subsection{}\label{sec2.4}

The punchline of the previous ideas around Kerov's mixing construction is that both Proposition \ref{prop2.C} and \ref{prop2.D} allow us to define new functionals in $\Phi_1(t)$ from known ones, at least when $t\in(0,1)$.

\begin{proposition}\label{prop2.E}
Fix $t\in(0,1)$ and introduce two multiplicative unital linear functionals, denoted by $\varphi_\row$ and $\varphi_\col$, by setting
\begin{equation*}
\varphi_\row(p_k)=1, \quad \varphi_\col(p_k)=(-1)^{k-1}\frac{(1-t)^k}{1-t^k},
\end{equation*}
for all $k=1,2,\cdots$. Then we have 
$$
\varphi_\row(P_\la(\ccdot;t))=\begin{cases} 1, & \text{if }\la=(n),\text{ for some }n=0,1,2,\dots, \\ 0, & \text{\rm otherwise,} \end{cases}
$$
$$
\varphi_\col(Q_\la(\ccdot;t))=\begin{cases} (1-t)^n, & \text{if }\la=(1^n),\text{ for some }n=0,1,2,\dots, \\ 0, & \text{\rm otherwise.} \end{cases}
$$
\end{proposition}

\begin{proof}
The Cauchy identity \eqref{eq2.J} can be rewritten as the following identity in $\Sym^{\otimes2}[[u]]$:
\begin{equation}\label{eq2.K}
\sum_{\la\in\Y}P_\la(\ccdot;t)\otimes Q_\la(\ccdot;t) u^{|\la|}=\exp\left(\sum_{n\ge1}\frac{1-t^n}{n}p_n\otimes p_n\, u^n\right).
\end{equation}

Applying $\varphi_\row\otimes\mathrm{id}$ to both sides of \eqref{eq2.K}, we get
\begin{gather*}
\sum_{\la\in\Y}\varphi_\row(P_\la(\ccdot;t)) Q_\la(\ccdot;t) u^{|\la|}=\exp\left(\sum_{n\ge1}\frac{1-t^n}{n}p_n\, u^n\right)\\
=\frac{H(u)}{H(tu)}=\sum_{n\ge0}Q_{(n)}(\ccdot;t) u^n,
\end{gather*}
where $H(u):=\exp\big( \sum_{n\ge 1}{\frac{1}{n}\,p_nu^n} \big)=\sum_{n\ge0}h_nu^n$, and the last equality is \cite[Ch.~III, (2.10)]{Mac}. This proves the first formula.

Likewise, applying $\mathrm{id}\otimes\varphi_\col$ to both sides of \eqref{eq2.K}, we get
\begin{gather*}
\sum_{\la\in\Y} \varphi_\col(Q_\la(\ccdot;t)) P_\la(\ccdot;t)u^{|\la|}=\exp\left(\sum_{n\ge1}(-1)^{n-1}\frac{(1-t)^n}{n}p_n\, u^n\right)\\
=\sum_{n\ge0}{e_n\big((1-t)u\big)^n}=\sum_{n\ge0}{P_{(1^n)}(\ccdot;t) \big((1-t)u\big)^n}.
\end{gather*}
This proves the second formula.
\end{proof}

\medskip

We can now apply Kerov's mixing construction (Proposition \ref{prop2.C}) starting from the simplest functionals $\varphi_\row$ and $\varphi_\col$, as follows.

\smallskip

Let $t\in(0,1)$, let $r,s\ge 1$ be any integers, and let $\alpha_1\ge\dots\ge\alpha_r\ge 0$ and $\beta_1\ge\dots\ge\beta_s\ge 0$ be real numbers such that
\[
\sum_{i=1}^r{\al_i} +\frac{1}{1-t}\sum_{j=1}^s{\be_j}=1.
\]
Further, let $\varphi^{(1)},\dots,\varphi^{(r)}$ be $r$ copies of $\varphi_\row$ and let $\psi^{(1)},\dots,\psi^{(s)}$ be $s$ copies of $\varphi_\col$. Then the functional
\begin{equation}\label{eq2.P}
\varphi^{(1)}_{\alpha_1}*\dots*\varphi^{(r)}_{\alpha_r}*\psi^{(1)}_{\frac{\beta_1}{1-t}}*\dots*\psi^{(s)}_{\frac{\beta_s}{1-t}},
\end{equation}
belongs to $\Phi_1(t)$, by virtue of Proposition \ref{prop2.C}.

Moreover, $\Delta^{r+s}(p_k) = \sum_{i=1}^{r+s}{1\otimes\dots\otimes 1\otimes p_k\otimes 1\otimes\dots\otimes 1}$, for all $k\ge 1$, where the summand indexed by $i$ has $p_k$ at the $i$-th position of the pure tensor.
Thus, from the explicit formula \eqref{eq2.O}, the values at $p_k$ of our functional in \eqref{eq2.P} coincide with the values at $p_k$ of the functional in $\ex(\Phi_1(t))$ from Theorem~\ref{thm1.A}, when the number of $\alpha$ and $\beta$ parameters is finite.
In addition, if $\varphi,\psi$ are multiplicative functionals, and $r\ge 0$, then $\varphi*\psi$ and $\varphi_r$ are also multiplicative, for any $r\ge 0$, as can be verified from the definitions \eqref{eq2.C}-\eqref{eq2.M}.
Hence, since $\varphi_\row$ and $\varphi_\col$ are multiplicative, then so is \eqref{eq2.P}; therefore, this functional coincides exactly with the functional in $\ex(\Phi_1(t))$ from Theorem~\ref{thm1.A}.

In fact, when $t\in(0,1)$, even the most generic functional in $\ex(\Phi_1(t))$ from Theorem~\ref{thm1.A} --- namely, the one with infinitely many parameters $\alpha_1\ge\alpha_2\ge\dots\ge 0$, $\be_1\ge\be_2\ge\dots\ge 0$, satisfying the more general condition $\sum_{i=1}^\infty{\alpha_i}+\frac{1}{1-t}\sum_{j=1}^\infty{\be_j}\le 1$ --- can be obtained as a limit of the ones in \eqref{eq2.P}, constructed by means of Kerov's construction, by virtue of Proposition~\ref{prop2.D}.

\smallskip

The conclusion is that Kerov's mixing construction leads to a very large family of functionals belonging to $\Phi(t)$, when $t\in(0,1)$.
And in fact, the functionals obtained exhaust the set $\ex(\Phi_1(t))$, as shown by Matveev~\cite{Mat}.
In Section~\ref{sect5}, we show an analogous construction that also leads to a large subset of $\Psi(-t)$, i.e. to a large family of $p_2$-harmonic, $(-t)$-HL-positive functionals of $\Sym$.

\section{The branching graphs $\YHL_\even(-t)$ and $\YHL_\odd(-t)$, and a theorem of Shen and Van Peski}\label{sect3}

\subsection{}

We modify the relation $\mu\nearrow\la$ as follows.

\begin{definition}[\cite{CO1}, Definition~8.1]\label{def3.A}
For $\nu,\la\in\Y$, we write $\nu\rel\la$ if $\nu\subset\la$, $|\la|=|\nu|+2$, and the two boxes of the skew diagram $\la/\nu$ either lie in a single column (so that $\la/\nu$ is a vertical domino) or lie in two consecutive columns (in particular, $\la/\nu$ may be a horizontal domino).
\end{definition}

Recall the notation $\wt P_\la(\ccdot;-t):=(-1)^{n(\la)}P_\la(\ccdot;-t)$ from \eqref{eq1.A}.

\begin{proposition}[\cite{CO1}, Proposition~8.2]\label{prop3.A}
Let $t\in(0,1)$. One has 
\begin{equation}\label{eq3.A}
p_2\cdot \wt P_\nu(\ccdot;-t)=\sum_{\la:\, \nu\rel\la} w(\nu\rel\la;-t) \wt P_\la(\ccdot;-t),
\end{equation}
where $w(\nu\rel\la;-t)$ are some strictly positive coefficients given by explicit formulas. 
\end{proposition}

The set of all partitions can be expressed as the disjoint union $\Y=\Y_\even\sqcup\Y_\odd$, where
$$
\Y_\even:=\Y_0\sqcup\Y_2\sqcup\Y_4\sqcup\cdots, \qquad \Y_\odd:=\Y_1\sqcup\Y_3\sqcup\Y_5\sqcup\cdots.
$$
Next, we form two graded rooted graphs: their vertex sets are $\Y_\even$ and $\Y_\odd$, respectively; the edges are formed by the pairs $\nu\rel\la$, and the roots are the empty diagram $\varnothing$ and the one-box diagram $(1)$, respectively. We denote these graphs by the same symbols as their vertex sets $\Y_\even$ and $\Y_\odd$. 

\begin{proposition}\label{prop3.B}
The graphs $\Y_\even$ and $\Y_\odd$ just defined are connected.
\end{proposition}

\begin{proof}
We verify that if $\la\in\Y$ is distinct from $\varnothing$ and $(1)$, then there exists $\nu\in\Y$ such that $\nu\rel\la$.
Recall that a Young diagram is said to be a \emph{$2$-core} if it has the form $(k,k-1,\dots,1)$, for some $k=1,2,\cdots$. If $\la$ is a $2$-core with $k>1$, then a desired $\nu$ does exists: one can remove from $\la$ any two neighboring  border boxes to get $\nu$. If $\la$ is not a $2$-core, then one can always remove from it a vertical or a horizontal domino.
\end{proof}

\begin{corollary}\label{cor3.A}
Let $t\in(0,1)$. We denote by $\YHL_\even(-t)$ and $\YHL_\odd(-t)$  the graded rooted graphs $\Y_\even$ and $\Y_\odd$, equipped with the edge weights $w(\nu\rel\la;-t)$ defined by equation \eqref{eq3.A} (cf. Definition~\ref{def2.D}). 

Then $\YHL_\even(-t)$ and $\YHL_\odd(-t)$ are branching graphs in the sense of Definition~\ref{def2.B}, and hence Proposition
~\ref{prop2.A} is applicable to them.  
\end{corollary}

\begin{proof}
Indeed, all conditions of Definition~\ref{def2.B} are satisfied. Namely, the edge weights are strictly positive (Proposition~\ref{prop3.A}), both graphs are connected (Proposition~\ref{prop3.B}), and there are no pending vertices (since for any $\nu$, the partition $\la$ obtained by adding a horizontal domino to the first row of $\nu$ satisfies $\nu\rel\la$).
\end{proof}

Let $\Sym_\even$ and $\Sym_\odd$ be the spans of the homogeneous components of $\Sym$ of even and odd degrees, respectively.
Recall that in Section~\ref{sec:1.4}, we defined $\Psi_\even(-t)$ and $\Psi_\odd(-t)$ as the convex cones of $p_2$-harmonic, $(-t)$-HL-positive linear functionals $\psi$ on $\Sym_\even$ and $\Sym_\odd$, respectively.
The cone $\Psi(-t)$ from Definition \ref{def1.B} is the direct sum of $\Psi_\even(-t)$ and $\Psi_\odd(-t)$. 

The cones $\Phi(\YHL_\even(-t))$ and $\Phi(\YHL_\odd(-t))$ of nonnegative harmonic functions, in the sense of Definition \ref{def2.C}, are naturally identified with $\Psi_\even(-t)$ and $\Psi_\odd(-t)$, respectively, by the equality $\psi(\wt P_\la(\cdot;-t))=\psi(\la)$.
Moreover, let 
\begin{eqnarray*}
(\Psi_\even)_1(-t) &:= \{ \psi\in\Psi_\even(-t) : \psi(1)=1 \},\\
(\Psi_\odd)_1(-t) &:= \{ \psi\in\Psi_\odd(-t) : \psi(p_1)=1 \}.
\end{eqnarray*}
Then $\Phi_1(\YHL_\even(-t))$ and $\Phi_1(\YHL_\odd(-t))$ are also identified with $(\Psi_\even)_1(-t)$ and $(\Psi_\odd)_1(-t)$, respectively.

By virtue of Proposition~\ref{prop2.A} and Corollary~\ref{cor3.A}, it follows that the convex cones $\Psi_{\even}(-t)$ and $\Psi_{\odd}(-t)$ are isomorphic to the cones of finite Borel measures on the sets of extreme points $\ex\big((\Psi_\even)_1(-t)\big)$ and $\ex\big((\Psi_\odd)_1(-t)\big)$, respectively.
Hence, explicit descriptions of these two sets would solve our Problem~\ref{problem1.A}.

\subsection{}

Let $\pi:\Sym\to\Sym$ denote the algebra endomorphism of plethysm with $p_2$, that is, $\pi(p_k)=p_{2k}$, for all $k=1,2,\cdots$.

Let $A$ and $B$ be two copies of $\Sym$. We regard $A$ as an algebra with respect to the usual operations in $\Sym$, while $B$ is endowed with the structure of an $A$-module coming from the map
\begin{equation}\label{eq3.C}
\na: A\otimes B \to B, \qquad (a,b)\mapsto \pi(a)b.
\end{equation}
The map $\na$ can be referred to as the \emph{$p_2$-twisted action} of $A$ on $B$.

\begin{definition}\label{def3.B}
We denote by $\wt f^{\,\la}_{\mu\nu}(t)$ the structure constants of $\na$ with respect to the basis $\big\{P_\mu(\ccdot;t^2): \mu\in\Y\big\}$ in $A$ and the basis $\{\wt P_\nu(\ccdot;-t): \nu\in\Y\}$ in $B$, in other words,
\begin{equation}\label{eq3.B}
\pi(P_\mu(\ccdot;t^2)) \wt P_\nu(\ccdot;-t)=\sum_{\la} \wt f^{\,\la}_{\mu\nu}(t) \wt P_\la(\ccdot;-t).
\end{equation}
Evidently, $\wt f^{\,\la}_{\mu\nu}(t)$ vanishes unless $|\la| =2|\mu|+|\nu|$. 
\end{definition}

The following fact is of key importance to us. It is extracted from the paper \cite{SVP1} by J.~Shen and R.~Van Peski.

\begin{theorem}\label{thm3.A}
Let $t=q^{-1}$, where $q=p^m$ is an odd prime power. Then the structure constants $\wt f^{\,\la}_{\mu\nu}(t)$ are nonnegative.
\end{theorem}

\begin{proof}
In \cite[(4.1)]{SVP1}, Shen and Van Peski introduce some structure constants $G^{\her,\la}_{\mu,\nu}(\oo)$, which are then renamed $g^{\her,\la}_{\mu,\nu}(q)$ (see \cite[Corollary 4.4]{SVP1}). By the very definition, these constants are nonnegative. 

On the other hand, \cite[Theorem 4.3]{SVP1} establishes an explicit connection between these constants and the constants that we denoted by $\wt f^{\,\la}_{\mu\nu}(t)$, when $t=q^{-1}$, making it clear that these two sorts of constants differ by positive factors. 
\end{proof}

Note that the main ingredient in the previous proof, namely \cite[Theorem 4.3]{SVP1}, is derived from results of Y. Hironaka \cite{Hir-1988}, \cite{Hir-1999}.

\section{Embeddings $\Phi(t^2)\to\Psi_\even(-t)$ and $\Phi(t^2)\to\Psi_\odd(-t)$ }\label{sect4}

Throughout this section, we assume that $t=q^{-1}$, where $q$ is an odd prime power. We need this assumption to be able to apply Theorem~\ref{thm3.A}.

\subsection{} 
Recall that, in our notation, $A$ and $B$ are two copies of $\Sym$. The difference between them is that we regard $A$ as an algebra and $B$ only as a vector space.

We equip $A$ with the HL inner product with parameter $t^2$, denoted by $\langle\ccdot,\ccdot\rangle_{t^2}$. The HL functions $P_\la(\ccdot;t^2)$ form an orthogonal basis of $A$, and the functions $Q_\la(\ccdot;t^2)=b_\la(t^2) P_\la(\ccdot;t^2)$ form the dual basis. Note that the numbers $b_\la(t^2)$ are strictly positive. 

We equip $B$ with the HL inner product with parameter $-t$, denoted by $\langle\ccdot,\ccdot\rangle_{-t}$.
The modified HL functions $\wt P_\la(\ccdot;-t):=(-1)^{n(\la)}P_\la(\ccdot;-t)$ form an orthogonal basis of $B$, and the functions $\wt Q_\la(\ccdot;-t)=b_\la(-t) \wt P_\la(\ccdot;-t)$ form the dual basis. Note that the numbers $b_\la(-t)$ are strictly positive. 

We split $B$ into an orthogonal direct sum, $B=B_\even\oplus B_\odd$, where $B_\even$ assembles the homogeneous components of $\Sym$ of even degree, while $B_\odd$ does the same for the homogeneous components of odd degree.

Recall that $\pi: \Sym\to\Sym$ is  the algebra endomorphism sending $p_k$ to $p_{2k}$ (plethysm with $p_2$).

\subsection{}\label{sect4.2}

Let us treat $\pi$ as a linear map $A\to B_\even$ and denote by $\pi^*: B_\even \to A$ the adjoint map. According to \eqref{eq3.B}, we have 
\begin{equation}\label{eq4.A}
\pi(P_\mu(\ccdot;t^2))=\sum_{\la:\, |\la|=2|\mu|} \wt f^\la_{\mu\varnothing}(t) \wt P_\la(\ccdot;-t).
\end{equation}
Therefore, for $\la\in\Y_\even$,
\begin{equation}\label{eq4.B}
\pi^*(\wt Q_\la(\ccdot;-t))=\sum\limits_{\mu:\,  |\mu|=|\la|/2}  \wt f^\la_{\mu\varnothing}(t)Q_\mu(\ccdot;t^2).
\end{equation}

Given a linear functional $\varphi: A\to\R$, we assign to it a linear functional $\psi_\even: B_\even\to\R$ by setting, for homogeneous elements $b\in B_\even$,
\begin{equation}\label{eq4.C}
\psi_\even(b):=\varphi(\pi^*(b)) 2^{-\frac12\deg b}.
\end{equation}

\begin{theorem}\label{thm4.A}
Assume that $t=q^{-1}$, where $q$ is an odd prime power.
The correspondence $\varphi\mapsto \psi_\even$ defined by equation \eqref{eq4.C} determines an embedding of the cone $\Phi(t^2)$ into the cone $\Psi_\even(-t)$.
It also restricts to an embedding of convex sets $\Phi_1(t^2)\hookrightarrow(\Psi_\even)_1(-t)$.
\end{theorem}

\begin{proof}
The map $\pi:A\to B_{\even}$ is injective, hence $\pi^*:B_{\even}\to A$ is surjective. Therefore, the linear map $\varphi\mapsto\psi_\even$ is injective. 

If $\varphi$ is $t^2$-positive, then we claim that $\psi_\even$ is $(-t)$-positive. Indeed, suffices to verify that $\psi_\even(\wt Q_\la(\ccdot;-t))\ge0$, for any $\la\in\Y_\even$. From \eqref{eq4.B} and \eqref{eq4.C}, it follows that 
$$
\psi_\even(\wt Q_\la(\ccdot;-t))=\sum\limits_{\mu:\,  |\mu|=\frac12|\la|}  \wt f^\la_{\mu\varnothing}(t)\varphi(Q_\mu(\ccdot;t^2))2^{-\frac12|\la|}.
$$
Since $\wt f^\la_{\mu\varnothing}(t)\ge0$ (by Theorem \ref{thm3.A}) and $\varphi(Q_\mu(\ccdot;t^2))\ge0$ (because $\varphi\in\Phi(t^2)$), we conclude that the whole expression is nonnegative, as desired.

It remains to prove that if $\varphi$ is $p_1$-harmonic, then $\psi_\even$ is $p_2$-harmonic, that is, $\psi_\even(p_2 b)=\psi_\even(b)$, for any $b\in B_\even$.
In fact, it suffices to prove this equality for all $b$ belonging to a basis of $B_\even$.
To this end, it will be convenient to work with the power sum symmetric functions
$$
p_\kappa=\prod_{i\ge1 }p_i^{m_i(\kappa)}, \quad \kappa=(1^{m_1(\kappa)} 2^{m_2(\kappa)}\dots)\in\Y,
$$
where $p_\varnothing=1$, by convention. These functions are orthogonal with respect to the HL inner product $\langle\ccdot,\ccdot\rangle_t$ for any value of the parameter $t$, and moreover
$$
\langle p_\kappa, p_\kappa\rangle_t=z_\kappa(t),
$$
where 
\begin{equation}\label{eq4.D}
z_\kappa(t):=z_\kappa \prod_{i\ge1}(1-t^i)^{-m_i(\kappa)}, \qquad z_\kappa:=\prod_{i\ge1} i^{m_i(\kappa)} m_i(\kappa)!,
\end{equation}
see \cite[Ch.~III, (4.1) and (4.11)]{Mac}. 
As a result, we have orthogonal bases $\{p_\rho: \rho\in\Y\}\subset A$ and $\{p_\si: \si\in\Y_\even\}\subset B_\even$, and the corresponding dual bases are 
$$
\left\{\frac{p_\rho}{z_\rho(t^2)}: \rho\in\Y\right \}\subset A, \qquad  \left\{\frac{p_\si}{z_\si(-t)}: \si\in\Y_\even\right\}\subset B_\even.
$$

By definition, $\pi(p_\rho)=p_{2\rho}$, for all $\rho\in\Y$, where $2\rho=(2\rho_1,2\rho_2,\dots)$. By duality, $\pi^*(p_\si)=0$, for all partitions $\si\in\Y_\even$ that are not of the form $2\rho$, while
$$
\pi^*\left(\dfrac{p_{2\rho}}{z_{2\rho}(-t)}\right)=\dfrac{p_\rho}{z_\rho(t^2)}.
$$ 
This last equality can be rewritten as
$$
\pi^*(p_{2\rho})=\dfrac{z_{2\rho}(-t)}{z_\rho(t^2)} p_\rho.
$$ 
Note that $m_j(2\rho)=0$, if $j$ is odd, while $m_{2i}(2\rho)=m_i(\rho)$, for all $i\ge 1$.
As a result, it follows from \eqref{eq4.D} that 
$$
\dfrac{z_{2\rho}(-t)}{z_\rho(t^2)}= \dfrac{z_{2\rho}}{z_{\rho}} = 2^{\ell(\rho)},
$$
and therefore
$$
\pi^*(p_{2\rho})=2^{\ell(\rho)} p_\rho .
$$ 
Taking into account the definition \eqref{eq4.C} and the fact that $\deg p_{2\rho}=2|\rho|$, we obtain
\begin{equation}\label{eq4.E}
\psi_\even(p_{2\rho})=\varphi(p_\rho) 2^{\ell(\rho)}2^{-|\rho|}.
\end{equation}

On the other hand, $p_2\cdot p_{2\rho}=p_{2(\rho\cup(1))}$ and the factor $2^{\ell(\rho)}2^{-|\rho|}$ does not change if $\rho$ is replaced by $\rho\cup(1)$.  These observations, together with definition \eqref{eq4.C} and equation \eqref{eq4.E} for $\rho\cup(1)$, give
$$
\psi_\even(p_2\cdot p_{2\rho})=\psi_\even(p_{2(\rho\cup(1))})=\varphi(p_{\rho\cup(1)})2^{\ell(\rho)}2^{-|\rho|}=\varphi(p_\rho\cdot p_1) 2^{\ell(\rho)}2^{-|\rho|}=\varphi(p_\rho) 2^{\ell(\rho)}2^{-|\rho|},
$$ 
where the last equality holds because $\varphi$ is $p_1$-harmonic. The final expression coincides with the right-hand side of \eqref{eq4.E}, proving that $\psi_\even(p_2\cdot p_{2\rho})=\psi_\even(p_{2\rho})$, for all $\rho\in\Y$.
Additionally, if $\sigma\in\Y_\even$ is not of the form $2\rho$, then $\psi_\even(p_\sigma)=\varphi(\pi^*(p_\sigma)) 2^{-\frac12|\sigma|}=0$, because $\pi^*(p_\sigma)=0$.
Likewise, $\psi_\even(p_2\cdot p_\sigma)=\psi_\even(p_{\sigma\cup(2)})=0$, because $\sigma\cup(2)$ also fails to be the form $2\rho$.
Hence, $\psi_\even(p_2b)=\psi_\even(b)$, for all $b$ belonging to the basis $\{p_\sigma : \sigma\in\Y_\even\}\subset B_\even$.

Finally, if $\varphi\in\Phi_1(t^2)$, the corresponding $\psi_{\even}$ is such that $\psi_{\even}(1)=\varphi(1)\cdot 2^0=1$, by virtue of \eqref{eq4.E} applied to $\rho=\emptyset$.
Hence, $\psi_{\even}\in(\Psi_\even)_1(-t)$, proving the last sentence, and henceforth the theorem.
\end{proof}

\subsection{}

Here, we obtain an analog of Theorem~\ref{thm4.A} for the cone $\Psi_\odd(-t)$. The argument is similar, with only minor modifications.

Let  $\wh\pi: A\to B_\odd$ be the linear map defined by
$$
\wh\pi(a):=p_1\cdot\pi(a), \quad a\in A,
$$
and let $\wh \pi^*: B_\odd \to A$ be the adjoint map. Since $p_1$ coincides with $P_{(1)}(\ccdot;t^2)=\wt P_{(1)}(\ccdot;t^2)$, we obtain from \eqref{eq3.B} that
\begin{equation}\label{eq4.A1}
\wh\pi(P_\mu(\ccdot;t^2))=\sum_{\la:\, |\la|=2|\mu|+1} \wt f^\la_{\mu(1)}(t) \wt P_\la(\ccdot;-t).
\end{equation}
Therefore, for $\la\in\Y_\odd$,
\begin{equation}\label{eq4.B1}
\wh\pi^*(\wt Q_\la(\ccdot;-t))=\sum\limits_{\mu:\, |\mu|=(|\la|-1)/2}  \wt f^\la_{\mu(1)}(t)Q_\mu(\ccdot;t^2).
\end{equation}

Given a linear functional $\varphi: A\to\R$, we assign to it the linear functional $\psi_\odd: B_\odd\to\R$ obtained by setting, for homogeneous elements $b\in B_\odd$,
\begin{equation}\label{eq4.C1}
\psi_\odd(b):=\varphi(\wh\pi^*(b)) 2^{-\frac12(\deg b-1)} (1+t).
\end{equation}

\begin{theorem}\label{thm4.B}
Assume that $t=q^{-1}$, where $q$ is an odd prime power.
The correspondence $\varphi\mapsto \psi_\odd$ defined by \eqref{eq4.C1} determines an embedding of the cone $\Phi(t^2)$ into the cone $\Psi_\odd(-t)$.
It also restricts to an embedding of convex sets $\Phi_1(t^2)\hookrightarrow(\Psi_\odd)_1(-t)$.
\end{theorem}

\begin{proof}
Both plethysm with $p_2$ and multiplication by $p_1$ are injective maps, therefore $\wh\pi:A\to B_\odd$ is also injective; as a result, the adjoint map $\wh\pi^*:B_\odd\to A$ is surjective. Therefore, the linear map $\varphi\mapsto\psi_\odd$ is injective. 

If $\varphi$ is $t^2$-positive, then $\psi_\odd$ is $(-t)$-positive: the argument is the same as for $\psi_\even$; we use the fact that the structure constants $\wt f^\la_{\mu(1)}(t)$ are nonnegative, by virtue of Theorem \eqref{thm3.A}. 

Now we have to prove that if $\varphi$ is $p_1$-harmonic, then $\psi_\odd$ is $p_2$-harmonic, that is, $\psi_\odd(p_2 b)=\psi_\odd(b)$, for any $b$ belonging to the orthogonal basis $\{p_\si: \si\in\Y_\odd\}\subset B_\odd$.
 
The duals of the orthogonal bases $\{p_\rho: \rho\in\Y\}\subset A$ and $\{p_\si: \si\in\Y_\odd\}\subset B_\odd$ are
$$
\left\{\frac{p_\rho}{z_\rho(t^2)}: \rho\in\Y\right \}\subset A, \qquad  \left\{\frac{p_\si}{z_\si(-t)}: \si\in\Y_\odd\right\}\subset B_\odd.
$$

By the definition of $\wh\pi$, it sends $p_\rho$ to $p_{(2\rho)\cup(1)}$, for all partitions $\rho$.
By duality, $\wh\pi^*(p_\si)=0$, for all partitions $\si\in\Y_\odd$ that are not of the form $(2\rho)\cup(1)$, while
$$
\wh\pi^*\left(\dfrac{p_{(2\rho)\cup(1)}}{z_{(2\rho)\cup(1)}(-t)}\right)=\dfrac{p_\rho}{z_\rho(t^2)}.
$$ 
The last equality can be rewritten as
$$
\wh\pi^*(p_{(2\rho)\cup(1)})=\dfrac{z_{(2\rho)\cup(1)}(-t)}{z_\rho(t^2)}p_\rho.
$$ 
Next,  from \eqref{eq4.D} and the fact that $1$ is not a part of the partition $2\rho$, we have $z_{(2\rho)\cup(1)}(-t) = z_{2\rho}(-t)z_{(1)}(-t)$; it follows that 
$$
\dfrac{z_{(2\rho)\cup(1)}(-t)}{z_\rho(t^2)}=z_{(1)}(-t)\dfrac{z_{2\rho}(-t)}{z_\rho(t^2)}=(1+t)^{-1}2^{\ell(\rho)}.
$$
As a result, we have (cf. \eqref{eq4.E})
\begin{equation}\label{eq4.F}
\psi_\odd(p_{(2\rho)\cup(1)})=\varphi(p_\rho)2^{\ell(\rho)}2^{-|\rho|}.
\end{equation}
Multiplying $p_{(2\rho)\cup(1)}$ by $p_2$ amounts to replacing $\rho$ by $\rho\cup(1)$, which does not affect the right-hand side.
Finally, equation~\eqref{eq4.F} also shows that if $\varphi(1)=1$, the corresponding $\psi_\odd$ satisfies $\psi_\odd(p_1)=1$.
As before, this completes the proof.
\end{proof}

\section{Analog of Kerov's mixing construction}\label{sect5}

\subsection{}

Recall that we defined in \eqref{eq3.C} the $p_2$-twisted multiplication map $\na:A\otimes B\to B$ using plethysm with $p_2$. Consider now the dual map  
$$
\wt\De: B \to A\otimes B,
$$
where duality is understood with respect to the HL inner product $\langle\ccdot,\ccdot\rangle_{t^2}$ in $A$ and the HL inner product $\langle\ccdot,\ccdot\rangle_{-t}$ in $B$. Let us emphasize that $\wt\De$ is different from the standard comultiplication map $\De:\Sym\to\Sym^{\otimes2}$. It can be called the \emph{$p_2$-twisted comultiplication}.

By duality, we obtain from \eqref{eq3.B} that
\begin{equation}\label{eq5.A}
\wt\De (\wt Q_\la(\ccdot;-t))=\sum_{\mu,\nu} \wt f^{\,\la}_{\mu\nu}(t)\, Q_\mu(\ccdot;t^2) \otimes \wt Q_\nu(\ccdot;-t).
\end{equation}

\subsection{}

Below, $\varphi: A\to\R$ and $\psi:B\to\R$ are linear functionals.
By analogy with \eqref{eq2.C}, we build from $\varphi$ and $\psi$ the new linear functional $\varphi\conv\psi: B\to\R$, defined by setting
\begin{equation}\label{eq5.B}
(\varphi\conv\psi)(b): =(\varphi\otimes \psi)(\wt\De(b)), \quad b\in B.
\end{equation}

\begin{proposition}\label{prop5.A}
Let $t=q^{-1}$, where $q$ is an odd prime power. If $\varphi$ is $t^2$-HL-positive and $\psi$ is $(-t)$-HL-positive, then $\varphi\conv\psi$ is $(-t)$-HL-positive.
\end{proposition}

\begin{proof}
This is a direct consequence of Theorem~\ref{thm3.A}. Indeed, it suffices to verify that $(\varphi\conv\psi)(\wt Q_\la(\ccdot;-t))\ge0$, for all $\la\in\Y$. By \eqref{eq5.A}, we have
$$
(\phi\conv\psi) (\wt Q_\la(\ccdot;-t))=\sum_{\mu,\nu} \wt f^{\,\la}_{\mu\nu}(t) \varphi(Q_\mu(\ccdot;t^2))\psi(\wt Q_\nu(\ccdot;-t)).
$$
The coefficients $\wt f^{\,\la}_{\mu\nu}(t)$ are nonnegative by virtue of Theorem \ref{thm3.A}, while $\varphi(Q_\mu(\ccdot;t^2))$ and $\psi(\wt Q_\nu(\ccdot;-t))$ are nonnegative by the assumptions. Hence, the sum above is also nonnegative, as desired.
\end{proof}

We define \emph{dilations} of $\varphi:A\to\R$ and $\psi:B\to\R$, with real parameters $r\ge 0$ and $s\ge 0$, respectively, as the linear functionals $\varphi_r:A\to\R$ and $\psi_s:B\to\R$, defined by the condition that for all homogeneous elements $ a\in A$ and $b\in B$,
$$
\varphi_r(a):=\varphi(a) r^{\deg a}, \qquad \psi_s(b):=\psi(b) s^{\frac12\deg b}.
$$
As before, we use the convention that $0^0=1$.

\begin{definition}[cf. Definition~\ref{def2.A}]\label{def5.A}
Let $\varphi:A\to\R$ and $\psi:B\to\R$ be two linear functionals. Their \emph{mixing with parameters $(r,s)$}, where $r,s\ge 0$, $2r+s=1$, is the linear functional $\varphi_r\conv\psi_s$.
\end{definition}

In more detail, the value of $\varphi_r\conv\psi_s$ on a homogeneous element $b\in B$ is given by the following formula.
For any decomposition $\wt\De(b)=\sum_i a_i\otimes b_i$, where $a_i\in A$ and $b_i\in B$ are homogeneous, then
\begin{equation*}
(\varphi_r\conv\psi_s)(b)=\sum_i \varphi(a_i)\psi(b_i)r^{\deg a_i}s^{\frac12\deg b_i}.
\end{equation*}
This formula is similar, but not identical, to \eqref{eq2.A}. The key difference comes from the fact that $\wt\De\ne\De$.

\begin{proposition}\label{prop5.B}
Let $t=q^{-1}$, where $q$ is an odd prime power. If $\varphi$ is $t^2$-HL-positive, $\psi$ is $(-t)$-HL-positive, and $r,s\ge 0$, then $\varphi_r\conv\psi_s$ is $(-t)$-HL-positive.
\end{proposition}

\begin{proof}
Note that if $\varphi$ is $t^2$-HL-positive and $r\ge 0$, then by definition $\varphi_r$ is $t^2$-HL-positive, too. Likewise, if $\psi$ is $(-t)$-HL-positive and $s\ge 0$, then $\psi_s$ is also $(-t)$-HL-positive. Then Proposition~\ref{prop5.A} finishes the proof.
\end{proof}

Next, we show that $\conv$ interacts well with the $p_2$-harmonicity property.

\begin{proposition}\label{prop5.C}
Suppose that $\varphi$ is $p_1$-harmonic and $\psi$ is $p_2$-harmonic. If $r,s\ge 0$ are such that $2r+s=1$, then $\varphi_r\conv\psi_s$ is $p_2$-harmonic.
\end{proposition}

\begin{proof}
As in the proof of Theorem~\ref{thm4.A}, we will work with the power sum symmetric functions
$$
p_\kappa=\prod_{i\ge1 }p_i^{m_i(\kappa)},\quad \kappa=(1^{m_1(\kappa)} 2^{m_2(\kappa)}\dots)\in\Y,
$$
where, by convention,  $p_\varnothing=1$. Recall that they are orthogonal with respect to the HL inner product $\langle\ccdot,\ccdot\rangle_t$, for any value of $t$, and
$$
\langle p_\kappa, p_\kappa\rangle_t=z_\kappa(t),
$$
where $z_\kappa(t)$ is defined in \eqref{eq4.D}.
As a result, the dual bases to the orthogonal bases $\{p_\rho: \rho\in\Y\}\subset A$, $\{p_\si: \si\in\Y\}\subset B$ are 
$$
\left\{\frac{p_\rho}{z_\rho(t^2)}: \rho\in\Y\right \}\subset A, \qquad  \left\{\frac{p_\si}{z_\si(-t)}: \si\in\Y\right\}\subset B.
$$

Note that $\pi(p_\mu)=p_{2\mu}$. Also, the operation of multiplication of power sum symmetric functions amounts to concatenation of partitions. Hence,
$$
\na(p_\mu\otimes p_\nu)=\pi(p_\mu)p_\nu=p_{(2\mu)\cup\nu}.
$$
From this, we deduce that
\begin{equation}\label{eq5.D}
\wt\De(p_\tau)=\sum_{\substack{\rho,\si\\ (2\rho)\cup\si=\tau}}\frac{z_\tau(-t)}{z_\rho(t^2)z_\si(-t)} p_\rho \otimes p_\si, \quad \tau\in\Y,
\end{equation}
and therefore
\begin{equation}\label{eq5.C}
(\varphi_r\conv\psi_s)(p_\tau)=\sum_{\substack{\rho,\si \\ (2\rho)\cup\si=\tau}}\frac{z_\tau(-t)}{z_\rho(t^2)z_\si(-t)}  
r^{|\rho|} s^{|\si|/2}\varphi(p_\rho)\psi(p_\si), \quad \tau\in\Y.
\end{equation}

We are going to show that under our assumptions, the right-hand side of \eqref{eq5.C} does not depend on the multiplicity $m_2(\tau)$ (with all other multiplicities $m_i(\tau)$ being fixed). This will imply that the right-hand side does not change when $\tau$ is replaced by $\tau\cup(2)$, which immediately implies that $\varphi_r\conv\psi_s$ is $p_2$-harmonic. 

Examine first the case when $\tau=(2^m)$. Then the condition $(2\rho)\cup\si=\tau$ means that $\rho=(1^k)$, $\si=(2^l)$, for some nonnegative integers $k,l$ such that $k+l=m$. By \eqref{eq4.D}, we have 
\begin{gather*}
z_{(2^m)}(-t)=2^m m! (1-(-t)^2)^{-m}=2^m m! (1-t^2)^{-m}, \\
z_{(1^k)}(t^2)=k! (1-t^2)^{-k}, \quad z_{(2^l)}(-t)=2^l l! (1-(-t)^2)^{-l}=2^l l! (1-t^2)^{-l}. 
\end{gather*}
It follows that the right-hand side of \eqref{eq5.C} in this case turns into
$$
\sum_{k+l=m}\frac{m!}{k! l!}\, 2^k r^k s^l \varphi(p_1^k) \psi(p_2^l)=\sum_{k+l=m}\frac{m!}{k! l!}\, (2r)^k s^l \varphi (p_1^k) \psi(p_2^l).
$$
Because $\varphi$ is $p_1$-harmonic and $\psi$ is $p_2$-harmonic, we have that $\varphi(p_1^k)=\varphi(1)$ and $\psi(p_2^l)=\psi(1)$. Thus, the previous expression simplifies to
$$
\varphi(1)\psi(1) \sum_{k+l=m} \frac{m!}{k! l!}\, (2r)^k s^l= \varphi(1)\psi(1) (2r+s)^m=\varphi(1)\psi(1).
$$
Thus, the result does not depend on $m$, as desired.

\smallskip

For the general case, we can write $\tau$ in the form $\bar\tau\cup (2^m)$, where $m_2(\bar\tau)=0$. Then, any splitting $\tau=(2\rho)\cup\si$ is determined by a splitting $\bar\tau=(2\bar\rho)\cup\bar\si$, together with a splitting of $(2^m)$, which is equivalent to a pair $(k,l)$ of nonnegative integers such that $m=k+l$.
In the splitting of $\bar\tau$, we have $m_1(\bar\rho)=0$ and $m_2(\bar\si)=0$.

Consequently, the sum over all possible couples $(\rho,\si)$ can be represented as a double sum: first, over $(\bar\rho,\bar\si)$ and next, over $(k,l)$. An important property of \eqref{eq4.D} is that we have the factorizations
$$
z_\tau(-t)=z_{(2^m)}(-t) z_{\bar\tau}(-t), \quad z_\rho(t^2)=z_{(1^k)}(t^2) z_{\bar\rho}(t^2), \quad z_\si(-t)=z_{(2^l)}(-t) z_{\bar\si}(-t). 
$$
This allows us to apply the above argument to the inner sum over $(k,l)$, for each fixed $(\bar\rho,\bar\si)$. We obtain that the inner sum depends only on $(\bar\rho,\bar\si)$, but not on $m$. Therefore, the whole expression does not depend on $m=m_2(\tau)$, as desired. This completes the proof. 
\end{proof}

\subsection{}

Let us assume that:

\begin{itemize}
\item $t=q^{-1}$, where $q$ is an odd prime power;

\item $\varphi: A\to\R$ is a linear functional, which is $t^2$-HL-positive and $p_1$-harmonic, that is,  $\varphi\in\Phi(t^2)$;

\item$\psi: B\to\R$ is a linear functional, which is $(-t)$-HL-positive and $p_2$-harmonic, that is, $\psi\in\Psi(-t)$;

\item $r$ and $s$ are two real parameters such that $r,s\ge 0$ and $2r+s=1$.

\end{itemize}
 
\begin{theorem}[Adaptation of Kerov's mixing construction for $\Psi(-t)$]\label{thm5.A}
Under these assumptions, the mixing $\varphi_r\conv\psi_s:B\to\R$, introduced in Definition~\ref{def5.A}, is $(-t)$-HL-positive and $p_2$-harmonic, that is, $\varphi_r\conv \psi_s\in\Psi(-t)$.
\end{theorem}

\begin{proof}
Follows from Propositions \ref{prop5.B} and \ref{prop5.C}. 
\end{proof}

Recall that $\Psi(-t)$ can be decomposed as $\Psi(-t)=\Psi_\even(-t)\oplus\Psi_\odd(-t)$, where
$$
\Psi_\even(-t):=\{\psi\in\Psi(-t): \psi\big|_{B_\odd}=0\},\quad \Psi_\odd(-t):=\{\psi\in\Psi(-t): \psi\big|_{B_\even}=0\}.
$$

The following corollary  of Theorem~\ref{thm5.A} is now evident. 

\begin{corollary}\label{cor5.A}
Under the same assumptions as above, we have:

{\rm(i)} If, additionally, $\psi\in\Psi_\even(-t)$, then $\varphi_r\conv \psi_s\in\Psi_\even(-t)$.

{\rm(ii)} If, additionally, $\psi\in\Psi_\odd(-t)$, then $\varphi_r\conv \psi_s\in\Psi_\odd(-t)$.
\end{corollary}

Next, recall that $(\Psi_\even)_1(-t)\subset\Psi_\even(-t)$ and $(\Psi_\odd)_1(-t)\subset\Psi_\odd(-t)$ serve as bases of the cones, and are determined by the normalizations $\psi(1)=1$ and $\psi(p_1)=1$, respectively.
As argued in Section~\ref{sect3}, finding a description of these convex sets (or of their sets of extreme points) would result in a complete description of $\Psi(-t)$ and would solve Problem~\ref{problem1.A}.
Our adapted Kerov's construction also yields new functionals in $(\Psi_\even)_1(-t)$ and $(\Psi_\odd)_1(-t)$ from known ones.

\begin{corollary}\label{cor5.B}
In addition to the assumptions above, assume that $\varphi\in\Phi_1(t^2)$. Then:

{\rm(i)} If, additionally, $\psi\in(\Psi_\even)_1(-t)$, then $\varphi_r\conv \psi_s\in(\Psi_\even)_1(-t)$.

{\rm(ii)} If, additionally, $\psi\in(\Psi_\odd)_1(-t)$, then $s^{-\frac12}\psi_s$ is a well-defined functional in $(\Psi_\odd)_1(-t)$, even when $s=0$ (using the convention that $0^0=1$), and $\varphi_r\conv s^{-\frac12}\psi_s\in(\Psi_\odd)_1(-t)$.
\end{corollary}

\begin{proof}
We argue first that $s^{-\frac12}\psi_s$ is a valid functional in $(\Psi_\odd)_1(-t)$, even when $s=0$.
Indeed, express $s^{-\frac12}\psi_s(b)=\psi(b)s^{\frac12 (\deg b-1)}$.
Then, since $b\in B_\odd$ implies $\deg b\ge 1$, we have that $s^{-\frac12}\psi_s(b)$ equals zero, when $s=0$, $\deg b>1$, and equals $\psi(b)$, when $s=0$, $\deg b=1$.

Finally, both (i) and (ii) are immediate consequences of the definition~\eqref{eq5.B} of the binary map $\conv$, together with $\wt\Delta(1)=1\otimes 1$ (for (i)) and $\wt\Delta(p_1)=1\otimes p_1$ (for (ii)).
\end{proof}

\begin{remark}
Theorems \ref{thm4.A} and \ref{thm4.B} can be obtained from Theorem \ref{thm5.A} as limit cases.
Specifically, let $\varphi\in\Phi(t^2)$ and $\psi_\even\in\Psi_\even(-t)$ be as in Theorem \ref{thm4.A}.
Next, let $\psi: B_\even\to\R$ be an \emph{arbitrary} linear functional such that $\psi(1)=1$.
Then 
$$
\lim_{s\to +0} (\varphi_{\frac{1-s}2}\conv\psi_s)(b)=\psi_\even(b), \quad b\in B_\even.
$$

Likewise,  let $\varphi\in\Phi(t^2)$ and $\psi_\odd\in\Psi_\odd(-t)$ be as in Theorem \ref{thm4.B}. Next, let $\psi: B_\odd\to\R$ be an \emph{arbitrary} linear functional such that $\psi(p_1)=1$. Then 
$$
\lim_{s\to +0} (\varphi_{\frac{1-s}2}\conv s^{-\frac12}\psi_s)(b)=\psi_\odd(b), \quad b\in B_\odd.
$$
\end{remark}

\section{The action of the twisted comultiplication map $\wt\De$}\label{sect6}

\subsection{Interaction between $\De$ and $\wt\De$}

Let ${\mathrm{m}}$ denote the standard multiplication in $\Sym$ viewed as a map $A^{\otimes2}\to A$. Consider the map $\xi:={\mathrm{m}}\circ\De$, which is the composition 
$$
A \to A^{\otimes 2} \to A.
$$
It is an algebra morphism sending each $p_k$ to $2p_k$. Recall that $\pi: A\to A$ is our notation for another algebra morphism, which sends $p_k$ to $p_{2k}$. 

Below, we use Sweedler's shorthand notation for comultiplication. Thus, for any $a\in A$ and $b\in B$, we will write
\begin{equation}
\De (a)=a_{(1)}\otimes a_{(2)} \in A^{\otimes2}, \qquad \wt\De(b)=b_{(1)} \otimes b_{(2)} \in A\otimes B.
\end{equation}

\begin{theorem}\label{thm6.A}
In this notation, 
\begin{equation}\label{eq6.A}
\wt\De(\pi(a)b)=\xi(a_{(1)})b_{(1)} \otimes \pi(a_{(2)}) b_{(2)}.
\end{equation}
\end{theorem}

\begin{proof}
It suffices to prove the equality when both $a,b$ are products of power sum symmetric functions.
We divide the proof into steps.

\smallskip

\emph{Step 1.} First, let us prove it in the special case when
$$
a=a_n:=(p_k)^n, \qquad b=b_m:=(p_{2k})^m,
$$ 
for a fixed $k\ge 1$, and arbitrary $m,n\ge 0$.

Since $a_n=a_1^n$ and $\De(a_1)=a_1\otimes 1+1\otimes a_1$, we have
\begin{equation}\label{eq6.B1}
\De(a_n)=\De(a_1^n)=\De(a_1)^n=(a_1\otimes 1+1\otimes a_1)^n=\sum_{\al\ge0}\binom n\al a_\al\otimes a_{n-\al},
\end{equation}
where the second equality follows because $\De$ is an algebra homomorphism and the last one is due to the binomial formula.
Note that omitting the upper limit of summation is correct because $\binom{n}\al$ automatically vanishes when $\al>n$.

Next, we compute $\wt\De(b_m)=\wt\De\big((p_{2k})^m\big)$ from equation~\eqref{eq5.D}.
Note that the only partitions $\rho,\sigma$ such that $2\rho\cup\sigma=(2k)^m$ are $\rho=(k^\beta)$ and $\sigma=\big((2k)^{m-\beta}\big)$, for some $0\le\beta\le m$.
Therefore
\begin{equation}\label{eq6.D}
\wt\De(b_m) = \wt\De\big((p_{2k})^m\big) = \sum_{\beta=0}^m{ \frac{z_{(2k)^m}(-t)}{z_{(k^\beta)}(t^2)z_{(2k)^{m-\beta}}(-t)}\, a_\beta\otimes b_{m-\beta} }.
\end{equation}
From definition~\eqref{eq4.D},
\begin{gather*}
z_{(2k)^m}(-t) = (2k)^m m!\cdot (1 - (-t)^{2k})^{-m},\\
z_{(k^\beta)}(t^2) = k^\beta\beta!\cdot\big(1 - (t^2)^k\big)^{-\beta},\qquad
z_{(2k)^{m-\beta}}(-t) = (2k)^{m-\beta} (m-\beta)!\cdot (1 - (-t)^{2k})^{-m+\beta}.
\end{gather*}
Plugging these values back into equation~\eqref{eq6.D}, we obtain
\begin{equation}\label{eq6.B2}
\wt\De(b_m) = \sum_{\be\ge0} \binom m\be 2^\be a_\be \otimes b_{m-\be},
\end{equation}
where the sum ranges over all $\beta\ge 0$ because ${\binom m\be}=0$, as soon as $\be>m$.

\smallskip

Let us go back to~\eqref{eq6.A} that we want to prove.
The previous equation~\eqref{eq6.B2} shows that the left-hand side of \eqref{eq6.A} is equal to
\begin{equation}\label{eq6.B3}
\wt\De\big(\pi(a_n)b_m\big)=\wt\De\big(b_nb_m\big)=\wt\De(b_{n+m})=\sum_{\ga\ge0} \binom {n+m}\ga 2^\ga a_\ga \otimes b_{n+m-\ga}.
\end{equation}
By Vandermode's identity,
$$
\binom {n+m}\ga=\sum_{\substack{\al,\be\ge 0\\ \al+\be=\ga}}\binom n\al \binom m\be,
$$
so \eqref{eq6.B3} can be rewritten as
\begin{equation}\label{eq6.B4}
\wt\De(\pi(a_n)b_m)=\sum_{\al\ge0}\sum_{\be\ge0}  \binom n\al \binom m\be 2^{\al+\be} a_{\al+\be} \otimes b_{n+m-\al-\be}.
\end{equation}
This is the left-hand side of \eqref{eq6.A} in our particular case.

On the other hand, by the expansions \eqref{eq6.B1} and \eqref{eq6.B2} of $\De(a_n)$ and $\wt\De(b_m)$, respectively, the right-hand side of \eqref{eq6.A} is equal to
\begin{equation}\label{eq6.B5}
\sum_{\al\ge 0}\sum_{\be\ge 0}{ \binom n\al \binom m\be 2^\be \xi(a_\al) a_\be \otimes \pi(a_{n-\al}) b_{m-\be} }.
\end{equation}
By the definitions, $\xi(a_\al)a_\be=2^\al a_\al a_\be=2^\al a_{\al+\be}$, and $\pi(a_{n-\al})b_{m-\be}=b_{n-\al}b_{m-\be}=b_{n+m-\al-\be}$.
By plugging these equalities into \eqref{eq6.B5}, we see that the resulting expression matches \eqref{eq6.B4}.
Hence, the desired \eqref{eq6.A} is proved in our special case.

\smallskip

\emph{Step 2.} Next, we prove equation \eqref{eq6.A} when
\begin{equation*}
a=\prod_{k\ge 1}{(p_k)^{n_k}}, \qquad b=\prod_{k\ge 1}{(p_{2k})^{m_k}},
\end{equation*}
for nonnegative integers $n_k$, $m_k$, of which only finitely many are nonzero.

If $\tau'$ and $\tau''$ have no common parts, then by \eqref{eq4.D}, it is evident that $z_{\tau'\cup\tau''}(\cdot)=z_{\tau'}(\cdot) z_{\tau''}(\cdot)$; also, $p_{\tau'\cup\tau''}=p_{\tau'}p_{\tau''}$.
Moreover, the partitions in a pair $(\rho,\sigma)$ satisfying $2\rho\cup\sigma=\tau$ can be uniquely decomposed as $\rho=\rho'\cup\rho''$ and $\sigma=\sigma'\cup\sigma''$, in such a way that $2\rho'\cup\sigma'=\tau'$ and $2\rho''\cup\sigma''=\tau''$.
Consequently, from the formula \eqref{eq5.D} for $\wt\De$, we deduce that
\begin{equation}\label{eq6.C}
\wt\De(p_{\tau'\cup\tau''}) = \wt\De(p_{\tau'})\wt\De(p_{\tau''}),\text{ if $\tau',\tau''$ have no common parts.}
\end{equation}
From Step 1, equation \eqref{eq6.A} holds for $a=(p_k)^{n_k}$, $b=(p_{2k})^{m_k}$.
By multiplying these equalities, over all $k\ge 1$, and making use of \eqref{eq6.C}, we find that \eqref{eq6.A} holds for the desired $a, b$.

\smallskip

\emph{Step 3.} Here we verify \eqref{eq6.A} for $a=1$ and arbitrary $b$.
Since $\pi(1)=1$, the left-hand side of \eqref{eq6.A} is $\wt\De(\pi(1)b)=\wt\De(b)$.
Also, $\Delta(1)=1\otimes 1$ implies that $\xi(1)=1$ and that the right-hand side of \eqref{eq6.A} equals $\xi(1)b_{(1)} \otimes \pi(1) b_{(2)}=b_{(1)}\otimes b_{(2)}$.
Hence, both sides agree.

\smallskip

\emph{Step 4.} Finally, the most general case is when
$$
a=\prod_{k\ge 1}{(p_k)^{n_k}}, \qquad b=\prod_{k\ge 1}{(p_{2k})^{m_k}}\cdot\bar b,
$$
for nonnegative integers $n_k$, $m_k$, of which only finitely many are nonzero, and $\bar b$ is a monomial in $p_1, p_3, p_5,\dots$ (odd indices). 
In Step~2, we proved the equality \eqref{eq6.A} for $a=\prod_{k\ge 1}{(p_k)^{n_k}}$, $b=\prod_{k\ge 1}{(p_{2k})^{m_k}}$, while Step~3 proves it for $a=1$, $b=\bar b$.
By multiplying them, and using~\eqref{eq6.C}, the desired equality follows.
\end{proof}

\begin{remark}
Setting $a=p_1$ in \eqref{eq6.A} leads to another proof of Proposition~\ref{prop5.C}.
\end{remark}

\begin{remark}
Recall that $B$ is an $A$-module with respect to the map $\na$, that is, with respect to the action $a\cdot b=\na(a\otimes b)=\pi(a)b$, for $a\in A$, $b\in B$.
One can also endow $A\otimes B$ with the following $A$-module structure:
\begin{equation*}
a\cdot(a'\otimes b'):=\xi(a_{(1)})a'\otimes\pi(a_{(2)})b',\qquad a\in A,\quad a'\otimes b'\in A\otimes B,
\end{equation*}
where the right-hand side uses Sweedler's notation, and $\De (a)=a_{(1)}\otimes a_{(2)}$.
Then Theorem~\ref{thm6.A} can be interpreted as saying that $\wt{\Delta}:B\to A\otimes B$ is an $A$-module homomorphism.
\end{remark}

\subsection{Analogies with van Leeuwen's and our previous work}

The constructions and results of this section (more generally, of this paper) were motivated by our research on measures on spaces of infinite matrices that are invariant with respect to the action of $\U(2\infty,\F_{q^2})$, where $q$ is an odd prime power.
We showed in~\cite{CO2} that the problem of interest is equivalent to the classification of positive harmonic functionals on the space $B_{q^2}=\bigoplus_{n=0}^\infty{\mathcal{C}(\mathfrak{u}(2n,\F_{q^2}))}$ of invariant functions on all Lie algebras $\mathfrak{u}(2n, \F_{q^2})$, $n\in\Z_{\ge 0}$.
Then we studied this classification problem by appealing to the module and comodule structures of $B_{q^2}$ with respect to the space $A_{q^2}=\bigoplus_{n=0}^\infty{\mathcal{C}(\mathfrak{gl}(n,\F_{q^2}))}$ of invariant functions on all $\mathfrak{gl}(n, \F_{q^2})$, $n\in\Z_{\ge 0}$.

The analogy between our aforementioned work and the results in the present paper is suggestive of a deeper link.
In fact, Theorem~\ref{thm6.A} is the analogue of \cite[Theorem 9.2]{CO2}, which describes what we call the ``twisted bimodule'' structure of $B_{q^2}$ with respect to $A_{q^2}$.
Likewise, Theorem~\ref{thm5.A} is the analogue of \cite[Theorem 7.3]{CO2} --- this is some sort of Kerov's mixing construction.

It seems plausible that this is more than a similarity: our work is kind of a translation of the construction of~\cite{CO2} to the language of symmetric functions, like a characteristic map.
More explicitly, if we restrict the setting of our note~\cite{CO2} to the subspaces $B_{q^2}^0\subset B_{q^2}$ and $A_{q^2}^0\subset A_{q^2}$, consisting of invariant functions supported on nilpotent matrices, then it is plausible that $B_{q^2}^0$ and $A_{q^2}^0$ can be identified with our $B_\even$ and $A$, respectively, in such a way that $\na:A\otimes B_\even\to B_\even$ and $\wt\Delta:B_\even\to A\otimes B_\even$ exactly coincide with the module and comodule structures of representation-theoretic origin from our previous work.
Let us point out here that in another recent paper, Shen and Van Peski prove a similar result in the related setting of abelian $p$-groups; see~\cite[Theorem~1.1 (Hermitian case)]{SVP2}.

Our note~\cite{CO2} itself was motivated by the work of van Leeuwen~\cite{L} that proves a version of \emph{Mackey's formula}, an identity relating the functors of parabolic induction and restriction between certain categories of finite Lie group representations.
This leads to an identity involving induced and restricted characters of Lie group representations.
Our \cite[Theorem~9.2]{CO2} is the parallel result for functions on Lie algebras that are invariant with respect to their Lie group actions.
Hence, Theorem~\ref{thm6.A} is the symmetric function version of Mackey's formula, which strips the representation-theoretic origin of the maps $\na$ and $\wt\Delta$.

\section{Final remarks}\label{sect7}

\subsection{}

The standard coproduct $\Delta:\Sym\to\Sym^{\otimes 2}$ cannot be used in the Kerov-type construction of Proposition~\ref{prop5.C} (that uses instead the $p_2$-twisted $\wt\Delta$) to produce new functionals in $\Psi(-t)$ from old ones.

Indeed, recall that the key statement for the original Kerov's construction from Section~\ref{sec2.3} was Lemma~\ref{lem2.A}, which states that the structure constants $f^\la_{\mu\nu}(t)$ of $\Delta$ with respect to the basis $\{Q_\la(;t):\la\in\Y\}\subset\Sym$ are nonnegative, for all $t\in (0,1)$.
The basis of $Q$-Hall-Littlewood symmetric functions was the chosen one because the positivity condition defining functionals $\varphi$ in $\Phi_1(t)$ is equivalent to $\varphi(Q_\la(\cdot;t))\ge 0$, for all $\la\in\Y$.
And actually, in the main application of Kerov's construction discussed in Section~\ref{sec2.4}, we only used that the structure constants $f^\la_{\mu\nu}(t)$ are nonnegative in the special cases when $\nu$ is a row or a column partition.

The positivity condition defining functionals $\psi$ in $\Psi(-t)$ is equivalent to $\psi(\wt{Q}_\la(\cdot;-t))\ge 0$, for all $\la\in\Y$, where $\wt{Q}_\la(;-t)=(-1)^{n(\la)}Q_\la(;-t)$.
The structure constants for the standard coproduct $\Delta$ with respect to this basis are the unique values $\wt{f}^\la_{\mu\nu}(t)$ such that
\[
\Delta\big( \wt{Q}_\la(;-t) \big) = \sum_{\mu,\nu}{ \wt{f}^\la_{\mu\nu}(t) \wt{Q}_\mu(;-t)\otimes\wt{Q}_\nu(;-t) },
\]
for all $\la$. They are equal to
\begin{equation}\label{eq7.A}
\wt{f}^\la_{\mu\nu}(t) = (-1)^{n(\la)-n(\mu)-n(\nu)}f^\la_{\mu\nu}(-t),
\end{equation}
where $f^\la_{\mu\nu}(-t)$ are the structure constants with respect to $\{Q_\la(;-t):\la\in\Y\}$.
We check now that the values \eqref{eq7.A} are not necessarily positive, even when $\nu$ is a row or column partition.
This will prove our claim about the unsuitability of $\Delta$ for Kerov's construction applied to $\Psi(-t)$.

Indeed, note that if $\nu$ is a one-column or a one-row Young diagram of size $r$, meaning that $\nu=(1^r)$ or $\nu=(r)$, then $f^\la_{\mu\nu}(-t)$ is a Pieri coefficient for the basis of HL functions $\{P_\la(;-t):\la\in\Y\}$, for which explicit formulas are available in \cite[Ch.~III, (3.2) and (3.10)]{Mac}.
These formulas show that $f^\la_{\mu\nu}(-t)$ is nonzero if and only if the skew diagram $\theta:=\la\setminus\mu$ is a vertical or horizontal $r$-strip, and moreover $f^\la_{\mu\nu}(-t)\ge 0$, whenever $t\in(0,1)$.
Therefore, if $\wt{f}^\la_{\mu\nu}(-t)$ is nonzero, then it has the same sign as
\[
(-1)^{n(\la)-n(\mu)-n(\nu)} = (-1)^{n(\theta)}\cdot (-1)^{n(\nu)}.
\]
Whether $\nu=(1^r)$ or $\nu=(r)$, the sign $(-1)^{n(\nu)}$ is determined, so the sign of $\wt{f}^\la_{\mu\nu}(-t)$ depends only on the parity of $n(\theta)=\sum_{i\ge 1}{(i-1)\theta_i}$.
However, for a \emph{skew} strip $\theta$, the quantity $n(\theta)$ can be both even or odd, depending on the form of $\theta$, so that the sign of $\wt{f}^\la_{\mu\nu}(-t)$ can be positive or negative, verifying our claim.

\subsection{}

Under the embedding $\Phi_1(t^2)\hookrightarrow(\Psi_\even)_1(-t)$ of Theorem~\ref{thm4.A}, the `Plancherel functional' from $\Phi_1(t^2)$, namely
$$
\varphi^{\text{Planch}}(p_\rho) = \begin{cases}
1,&\text{ if }\rho=(1^n),\text{ for some }n\in\Z_{\ge 0},\\
0,&\text{ otherwise},
\end{cases}
$$
is mapped to the `Plancherel functional' from $(\Psi_\even)_1(-t)$, namely
$$
\psi_{\even}^{\text{Planch}}(p_{2\rho}) = \begin{cases}
1,&\text{ if }\rho=(1^n),\text{ for some }n\in\Z_{\ge 0},\\
0,&\text{ otherwise}.
\end{cases}
$$
This follows from equation~\eqref{eq4.E}.
Likewise, equation~\eqref{eq4.F} shows that the embedding $\Phi_1(t^2)\hookrightarrow(\Psi_\odd)_1(-t)$ of Theorem~\ref{thm4.B} maps $\varphi^{\text{Planch}}$ to the `Plancherel functional' from $(\Psi_\odd)_1(-t)$:
$$
\psi_{\odd}^{\text{Planch}}(p_{2\rho\,\cup\,(1)}) = \begin{cases}
1,&\text{ if }\rho=(1^n),\text{ for some }n\in\Z_{\ge 0},\\
0,&\text{ otherwise}.
\end{cases}
$$
Both $\psi_{\even}^{\text{Planch}}$ and $\psi_{\odd}^{\text{Planch}}$ have appeared before in \cite[Section~8.4.1]{CO1}.

\bigskip

Cesar Cuenca:

\smallskip

${}^1$Department of Mathematics, The Ohio State University, Columbus, OH, USA.

\smallskip

Email address: cesar.a.cuenk@gmail.com

\bigskip

Grigori Olshanski:

\smallskip

${}^2$Higher School of Modern Mathematics, MIPT, Moscow, Russia;

\smallskip

${}^3$Skolkovo Institute of Science and Technology, Moscow, Russia;

\smallskip

${}^4$HSE University, Moscow, Russia.

\smallskip

Email address: olsh2007@gmail.com

\end{document}